\documentclass[11pt, hidelinks]{article}
\usepackage[margin =1in]{geometry}
\usepackage{amssymb,amsthm,amsmath}
\usepackage{enumerate}
\usepackage{hyperref}
\usepackage{cite}
\usepackage[capitalize]{cleveref}
\usepackage[shortlabels]{enumitem}
\usepackage{color}
\usepackage{cancel}
 \usepackage{tikz}
\usepackage{tikz-3dplot}
\usepackage{pgf}
\usepackage{svg}
\usepackage{pgfplots}
\usepackage{import}
\usepackage{svg}
\usepackage{comment}
\usepackage[utf8]{inputenc} 
\usepackage[T1]{fontenc}    
\usepackage{hyperref}       
\usepackage{url}            
\usepackage{booktabs}       
\usepackage{amsfonts}       
\usepackage{nicefrac}       
\usepackage{microtype}
\usepackage{multirow}
\usepackage{xfrac}
\usepackage{todonotes}
\usepackage{titlesec}
\usepackage{caption}
\usepackage{subcaption}

\usepackage{algorithm}
\usepackage{algpseudocode}
\usepackage{array}

\algtext*{EndFor}
\algtext*{EndIf}
\algtext*{EndWhile}
\algtext*{EndProcedure}
\algrenewtext{For}[1]{\textbf{for} #1}
\algrenewtext{While}[1]{\textbf{while} #1}
\algrenewtext{If}[1]{\textbf{if} #1}

\algnewcommand{\Input}{\item[\textbf{Input:}]}
\algnewcommand{\Initialize}{\item[\textbf{Initialize:}]}

\usepackage{varwidth}

\titleformat{\subsubsection}[runin]
{\normalfont\normalsize\bfseries}{\thesubsubsection}{1em}{}

\usepackage{graphicx}

\usepackage{appendix}
\numberwithin{equation}{section}

\newcommand{\spann}[1]{\mathrm{span}\left\{{#1}\right\}}

\theoremstyle{thmstyletwo}%
\newtheorem{thm}{Theorem}[section]
\newtheorem{theorem}[thm]{Theorem}

\newtheorem{proposition}[thm]{Proposition}

\newtheorem{lemma}[thm]{Lemma}

\newtheorem{corollary}[thm]{Corollary}

{\theoremstyle{plain}
}

\numberwithin{equation}{section}
\crefname{claim}{claim}{claims}
\Crefname{claim}{Claim}{Claims}
\crefname{lem}{lemma}{lemmas}
\Crefname{lem}{Lemma}{Lemmas}
\crefname{algorithm}{algorithm}{algorithms}
\Crefname{algorithm}{Algorithm}{Algorithms}

\newcommand{\R}{\mathbb{R}}
\newcommand{\N}{\mathbb{N}}

\newcommand{\Ical}{\mathcal{I}}

\newcommand{\inner}[2]{\left\langle {#1}, {#2} \right\rangle}
\newcommand{\Wopt}{\mathcal{W}_\star^{M,D,N}}

\usepackage{mathtools}

\usepackage[T1]{fontenc}
\usepackage{lmodern}
\newcommand{\pPEP}[1]{p_{#1}}

\newcommand{\Algclass}[1]{\mathtt{A}_{\mathrm{{#1}}}}

\newcommand{\Pclass}{\mathtt{P}_{M,D}}

\newlength{\solutionindent}
\newlength{\questionboxrule}
\usepackage[dvipsnames]{xcolor}

\usepackage{hyperref}

\hypersetup{
  pdftitle={
    A Complete Characterization of Optimal 
    Subgradient Methods for Lipschitz Convex Minimization
  },
  pdfauthor={
    Aaron Zoll and Benjamin Grimmer
  },
  pdfsubject={
    A characterization of all fixed-step first-order methods attaining
    the minimax-optimal objective-gap guarantee for Lipschitz convex
    minimization
  },
  pdfkeywords={
    nonsmooth convex optimization,
    Lipschitz convex minimization,
    subgradient methods,
    fixed-step first-order methods,
    minimax optimality,
    optimal first-order methods,
    performance estimation problems,
    worst-case analysis,
    polyhedral characterization,
    first-order oracle complexity
  },
  pdfcreator={LaTeX with hyperref},
  pdfdisplaydoctitle=true,
  unicode=true
}
\begin{document}
    \title{A Complete Characterization of Optimal Subgradient Methods for Lipschitz Convex Minimization}

	 \author{Aaron Zoll\footnote{Johns Hopkins University, Department of Applied Mathematics and Statistics, \url{azoll1@jhu.edu}} \qquad Benjamin Grimmer\footnote{Johns Hopkins University, Department of Applied Mathematics and Statistics, \url{grimmer@jhu.edu}}}

	\date{}
	\maketitle

\begin{abstract}
    We consider the design of optimal fixed-step first-order methods for $M$-Lipschitz convex optimization given $\|x_0-x_\star\|\leq D$. Prior works have identified several distinct fixed-step methods, parameterized by a matrix of stepsizes $W$, with the (information-theoretic) minimax optimal rate $MD/\sqrt{N+1}$ of objective gap convergence. We provide a complete characterization of every optimal fixed-step method. Moreover, we show every optimal fixed-step method can be derived from the constructive approach of~\cite{constructive_approach} and provide a polyhedral representation of the set of optimal methods through proof multipliers. From this characterization, we show that no anytime optimal fixed-step subgradient methods exist.
\end{abstract}

\section{Introduction}

We consider the classical problem of Lipschitz (potentially nonsmooth) convex minimization~\cite{Shor:Subgradient}. In arbitrary (potentially large) dimension $d$, we study problems of the form $$\min_{x \in \R^d} f(x).$$ Given problem parameters $M, D > 0$, a problem instance is a pair $(f, x_0)$ with $f$  convex, $M$-Lipschitz, and attaining a minimizer $x_\star$ with Euclidean norm $\|x_0-x_\star\| \leq D$. We denote the set of all instances of any dimension $d$ by $\Pclass$. Herein, we consider iterative methods given access to a first-order oracle mapping, in the $i$th iteration, $x_i$ to a pair $(f_i,g_i)$ with function value $f_i = f(x_i)$ and subgradient $g_i \in \partial f(x_i)$, lying in the subdifferential $\partial f(x)=\{g : f(y) \geq f(x)+\langle g, y-x\rangle\ \forall y\}$.

Given an iteration budget $N$, we consider the class of $N$-step subgradient span algorithms for which $x_n \in x_0 - \spann{g_0, \dots, g_{n-1}}$, denoted by $\Algclass{span}$. Given an algorithm, we measure its performance by its worst-case final objective gap over all problem instances. Optimal algorithm design then seeks methods with the best worst-case performance.

This ``minimax optimal'' rate of convergence is known to be exactly $\frac{MD}{\sqrt{N+1}}$. Namely,
\begin{equation}\label{eq:minimax-opt}
    \min_{a \in \Algclass{span}}\max_{(f, x_0) \in \Pclass} f(x_N)-f(x_\star) = \frac{MD}{\sqrt{N+1}}.
\end{equation}
Equality is attained by a simple averaged subgradient method, iterating from $y_0=x_0$ with $y_n = y_{n-1}-\frac{D}{M\sqrt{N+1}}g_{n-1}$ for $g_{n-1}\in\partial f(y_{n-1})$, reporting final iterate as the average $x_N = \frac{1}{N+1}\sum_{i=0}^N y_i$. This method has $f_N-f_\star \leq \frac{MD}{\sqrt{N+1}}$ by~\cite[Section 3.1]{bubeck2015convex} and no method can do better by the lower bound of~\cite[Theorem A.1]{drori2016optimal}.

This algorithm takes the form of a fixed-step first-order method (FSFOM). These methods, parameterized by a lower triangular matrix $W$, are defined by the iteration
\begin{equation*}
    x_n = x_0 - \sum_{i=0}^{n-1} W_{n, i} g_i
\end{equation*}
for $n=1,\dots,N$, where $g_i\in\partial f(x_i)$ are subgradients computed at each iteration. For ease of presentation, the rows of $W$ are indexed from $1$ to $N$ while the columns are indexed from $0$ to $N-1$.  Denote the set of all $N\times N$ lower triangular matrices by $\mathbb{L}^{N}$ and those with positive lower triangular entries by $\mathbb{L}^{N}_{++}$.
For example, the above averaged subgradient method corresponds to the matrix with constant $W_{i,j} = \frac{D}{M\sqrt{N+1}}$ for $0 \leq j < i \leq N-1$ and $W_{N,j} = \frac{D}{M\sqrt{N+1}}\frac{N-j}{N+1}$. 

Surprisingly, many minimax optimal fixed-step methods exist. Zamani and Glineur~\cite{zamani2025exact} provided an optimal subgradient method $x_n = x_{n-1}-h_{n-1} g_{n-1}$ with nonconstant stepsizes (but no final averaging). Their method is parameterized by the matrix with $W_{i,j} = \frac{D}{M\sqrt{N+1}} \frac{N-j}{N+1}$.  G\"osgens and Van Parys~\cite{gosgens2025adversarial} identified infinitely many optimal subgradient step methods which interpolate between the averaged method and the nonconstant stepsize method above. Separately, the constructive design approach of Drori and Taylor~\cite{constructive_approach} produced an optimal ``momentum-type'' method. Their algorithm, derived from their Subspace-Search Elimination Procedure (SSEP), is parameterized by the matrix with $W_{i,j} =  \frac{D}{M\sqrt{N+1}} \frac{i-j}{i+1}$.

Note that beyond the restricted family of fixed-step methods, other minimax optimal subgradient methods are known. For example, see the Kelley cutting-plane Like Method of~\cite{drori2016optimal}. In fact, this method attains a strong dynamic notion of optimality called subgame perfection~\cite{grimmer2025subgameperfectmethodsnonsmooth}.

Here we characterize all minimax optimal fixed-step methods, denoted $\Wopt$. First, we prove optimality for a considered family of algorithms. These are precisely the methods able to be generated by the constructive approach of~\cite{constructive_approach}. Doing so provides a convex (polyhedral) representation of a family of optimal methods. Complementing this, we show no other optimal methods exist, establishing this family as a complete characterization of all optimal fixed-step methods.

\subsection{Related Work and Preliminaries}

\paragraph{Performance Estimation.} The foundational work~\cite{drori2014performance} introduced Performance Estimation Problems (PEPs) as a principled way to analyze first-order methods over traditional problem classes. Subsequently,~\cite {taylor2017smooth,taylor2017composite} established the exactness of this tool via interpolation theory. For some $W$, a PEP finds the worst final objective gap $f(x_N)-f(x_\star)$ over all possible problem instances. Formally,
  \begin{equation}\label{eq:Lipshitz-FSFOM-PEP-true}
     \pPEP{\star}(W) = \begin{cases}\displaystyle\max & f(x_N)-f(x_\star)  \\
    \text{s.t.} & x_n = x_{0} - \sum_{i=0}^{n-1} W_{n,i} g_i \ \ \qquad \forall n = 1, 2, \dots, N\\
    & (f,x_0)\in \Pclass,\ g_i\in\partial f(x_i)\quad \forall i=0,1,\dots,N.
    \end{cases}
  \end{equation}

The key observation to tractably solve PEPs is that the algorithm's trajectory and performance depend only on the first-order information at the iterates $x_0,\dots,x_N$ and at the minimizer $x_\star$. For ease in referring to such points, we consider the index set $ \Ical_N^\star = \{\star, 0, 1, \dots, N\}$. The relevant quantities to algorithmic performance are then
\begin{equation}\label{eq:observations}
    x_i, \qquad f_i = f(x_i), \quad g_i\in\partial f(x_i) \qquad \forall i \in \Ical_N^\star
\end{equation}
where we require $g_\star=0$ at the minimizer $x_\star$. Following the interpolation theory of~\cite[Theorem 3.5]{taylor2017composite}, we know that for observations $\{(x_i, f_i, g_i)\}_{i \in \Ical_N^\star}$, there exists an $M$-Lipschitz, convex function, minimized at $x_\star$ for which~\eqref{eq:observations} holds if and only if the following values are nonnegative for all $i,j\in \Ical_N^\star$ 
\begin{align*}
    \mathcal{C}_{i,j} := f_i - f_j - \inner{g_j}{x_i-x_j}, \qquad
    \mathcal{M}_{i} := M^2 - \|g_i\|^2.
\end{align*}
Combined with the distance bound $\mathcal{D} := D^2-\|x_0-x_\star\|^2 \geq 0$, Taylor, Hendrickx, and Glineur~\cite{taylor2017smooth,taylor2017composite} showed one can define an equivalent optimization problem to~\eqref{eq:Lipshitz-FSFOM-PEP-true} over finitely many variables.

The standard approach to simplifying such PEP reformulations is to consider a Gram change of variables. Define $P = [x_0-x_\star \ | \ g_0 \ |\  g_1 \ |\  \dots\ | \ g_N]$ and the Gram matrix of all inner products between these vectors as $G=P^\top P$. After substituting the equality definitions of $x_n = x_{0} - \sum_{i=0}^{n-1} W_{n,i} g_i$ and $g_\star=0$ into the remaining inequality constraints, observe that $\mathcal{D}$, $\mathcal{C}_{i,j}$, and $\mathcal{M}_i$ are linear in $F=[f_\star,f_0,\dots, f_N]^\top$ and $G$. We denote these functions by $\mathcal{D}(F,G)$, $\mathcal{C}_{i,j}(F,G)$ and $\mathcal{M}_i(F,G)$. Additionally, as a Gram matrix, we require $G$ to be positive semidefinite. Changing to these variables results in a semidefinite program, denoted by
\begin{equation}\label{eq:PEP_gram}
    \pPEP{\star}(W) = \pPEP{gram}(W) = \begin{cases}\displaystyle\max_{F,G} & f_N-f_\star  \\
    \mathrm{s.t.} & \mathcal{D}(F,G) \geq 0 \\
    & \mathcal{C}_{i,j}(F,G) \geq 0 \quad  \forall i, j \in \Ical_N^\star\\
    & \mathcal{M}_i(F,G) \geq 0 \quad \forall i = 0, \dots, N\\
    & G\succeq 0.
    \end{cases}
\end{equation}
Although this reformulation is obtained by a relaxation, tightness (the equality above) follows from noting that any positive semidefinite $G$ can be factored into $P^\top P$ and any feasible $F,G$ can be used to construct an interpolating problem instance in dimension at most $d=N+2$.

We denote the dual variables to the constraints as nonnegative $\sigma, \lambda_{i,j}, \mu_i$ and positive semidefinite $Z$ respectively. Such multipliers prove a guarantee $f_N-f_\star \leq r$ if the identity in $F$ and $G$ of
\begin{equation}\label{eq:certificate}
    \sigma\mathcal D+\sum_{i,j \in \Ical_N^\star}\lambda_{i,j}\mathcal C_{i,j}+\sum_{i=0}^N\mu_i\mathcal{M}_i+\inner{Z}{G} = r - (f_N-f_\star)
\end{equation}
holds (substituting $x_n = x_{0} - \sum_{i=0}^{n-1} W_{n,i} g_i$ and $g_\star=0$). Note that $\mathcal{C}_{i,i}$ is identically zero, so we can fix $\lambda_{i,i}$ as zero for each index $i$ throughout without loss.

For any feasible $(F,G)$, the nonnegativity of the left-hand-side of the identity~\eqref{eq:certificate} then ensures the right-hand-side is nonnegative, proving the rate. The dual problem then minimizes $r$ over all feasible certificates $\sigma,\lambda,\mu,Z$. For an optimal algorithm and proof, this rate is $r=\frac{MD}{\sqrt{N+1}}$.

\begin{lemma}\label{lem:strong-duality}
     Given any $M, D > 0$ and any matrix $W\in\mathbb{L}^N_{++}$ parameterizing a FSFOM, there exist nonnegative $\sigma, \lambda_{i,j}, \mu_i$ and positive semidefinite $Z$ such that~\eqref{eq:certificate} holds with $r=p_\star(W)$.
\end{lemma}
\begin{proof}
    We prove this by establishing strong duality with dual attainment via a Slater point (strictly satisfying all non-affine constraints). Fix any orthonormal vectors $e_1,\dots,e_{N+2}$. Set $f_n=0$, $g_n=\frac{M}{2}e_{n+1}$, $x_n = x_0 - \sum_{i=0}^{n-1}W_{n,i}g_i$ for $n=0,\dots,N$, with $x_0=0$, $x_\star=-\frac{D}{2\sqrt{N+2}}\sum_{i=1}^{N+2}e_i$, $f_\star = -\frac{MD}{4\sqrt{N+2}}$, and $g_\star = 0$. Let $(F,G)$ be their Gram representation. By construction, $\|x_0-x_\star\|<D$ and $\|g_i\| <M$, so $\mathcal{D}(F,G), \mathcal{M}_i(F,G) > 0$. Note that $$\mathcal{C}_{i,j}(F,G) = f_i-f_j-\inner{g_j}{x_i-x_j}=\begin{cases}
        \frac{M^2}{4}W_{i,j} & i > j\\
        0 & \text{else,}
    \end{cases} \quad \mathcal{C}_{i, \star}(F,G) = \frac{MD}{4\sqrt{N+2}}, \quad \mathcal{C}_{\star, i}(F,G) = 0$$ 
    where we recall $f_i=f_j$, the method fixes $x_n=x_0-\sum_{i=0}^{n-1}W_{n,i}g_i$, and the gradients are orthogonal. Since $x_0-x_\star$ lies outside $\spann{e_1,\dots,e_{N+1}}$ and the gradients are orthogonal, $x_0-x_\star,g_0,\dots,g_N$ are linearly independent and $G\succ0$. Since the constraints $\mathcal{D}, \ \mathcal{C}_{i,j}, \ \mathcal{M}_i$ are all affine in $(F,G)$, and $G\succ0$, the feasible $(F,G)$ is a Slater point for~\eqref{eq:PEP_gram}. Since $p_\star(W)$ is finite, \cite[Theorem 28.2]{rockafellar} applies and strong duality holds with dual attainment. Therefore, there exist nonnegative $\sigma, \lambda_{i,j}, \mu_i$ and positive semidefinite $Z$ such that~\eqref{eq:certificate} holds with $r=p_\star(W)$.
\end{proof}

\paragraph{Converting From Stepsize Matrices to Proof Multipliers}
Recently, a complete characterization of all optimal nonexpansive fixed-point methods was given by~\cite{yoon2025}. A key step was identifying a relation between a matrix of stepsizes (there denoted $H$) and the associated proof multipliers $\lambda$. From this, they characterized every optimal fixed-point method based on associated polynomial equalities and inequalities, called $H$-invariances and $H$-certificates.

Here, we follow the same tactic. In particular, we will show that each optimal method has a unique dual optimal PEP proof, all using the same common values of
\begin{equation}\label{eq:common-values}
    \sigma = \frac{M}{2D\sqrt{N+1}}, \qquad \lambda_{\star,i} = \frac{1}{N+1}, \qquad \mu_{i} = \frac{D}{2M(N+1)^{3/2}}, \qquad Z=\sigma ww^\top
\end{equation}
with $w = [1, -\frac{D}{M\sqrt{N+1}}, \dots, -\frac{D}{M\sqrt{N+1}}]^\top$ and setting each $\lambda_{i,\star}$ and $\lambda_{i,j}$ with $i> j$ to zero. So optimal methods only differ in their optimal proof structure in the remaining $\lambda_{i,j}$ with $0 \leq i < j \leq N$. We find a bijection between these upper triangular proof multipliers $\lambda$ and the lower triangular stepsizes $W$. Define $W(\lambda)$ recursively by
\begin{equation}\label{eq:dual-generated-FSFOM}
W_{n,j}=\frac{1}{\Lambda_n}\left(\frac{D}{M(N+1)^{3/2}}
+\sum_{i=j+1}^{n-1}\lambda_{i,n}W_{i,j}\right),\qquad 0\leq j<n\leq N,
\end{equation}
where $\Lambda_n = \frac{1}{N+1}+\sum_{i=0}^{n-1}\lambda_{i,n}$, and define $\lambda(W)$ recursively by 
\begin{align}
\lambda_{j+1,n}&=\frac{1}{W_{j+1,j}}\left(\frac{DW_{n,j}}{M(N+1)^{3/2}W_{n,n-1}} -\frac{D}{M(N+1)^{3/2}}
-\sum_{i=j+2}^{n-1}\lambda_{i,n}W_{i,j}\right),\qquad j=n-2,\dots,0, \nonumber\\
\lambda_{0,n}&=\frac{D}{M(N+1)^{3/2}W_{n,n-1}}-\frac{1}{N+1}-\sum_{i=1}^{n-1}\lambda_{i,n}.\label{eq:lambda(W)}
\end{align}

These mappings are exactly each other's inverse on the domains given below. 
\begin{proposition}\label{prop:bijection}
Given $\lambda_{i,j}$ for $0 \leq i < j \leq N$ with $\Lambda_n=\frac{1}{N+1}+\sum_{i=0}^{n-1}\lambda_{i,n} > 0$,
and define the lower triangular $W=W(\lambda)\in \mathbb{L}^N$ by~\eqref{eq:dual-generated-FSFOM}.
Conversely, for lower triangular $W$ with $W_{n,n-1}>0$ define $\lambda=\lambda(W)$
by~\eqref{eq:lambda(W)}.
Then $W(\cdot)$ and $\lambda(\cdot)$ are inverses, providing bijections between
$$\left\{\lambda \mid \Lambda_n>0,\ n=1,\dots,N\right\}\quad\text{and}\quad
\left\{W \in \mathbb{L}^N \mid W_{n,n-1}>0,\ n=1,\dots,N\right\}.$$
\end{proposition}
\begin{proof}
We first verify the mappings are well-defined on the stated domains. At $j=n-1$ the sum in~\eqref{eq:dual-generated-FSFOM} is empty, so $W_{n,n-1}=\frac{D}{M(N+1)^{3/2}\Lambda_n}$, and hence $W_{n,n-1}>0$ if and only if $\Lambda_n>0$.

We now show these maps are inverses directly. Suppose $\Lambda_n>0$ for all $n$ and set $W=W(\lambda)$, which by the mapping in~\eqref{eq:dual-generated-FSFOM} has positive subdiagonal. Fixing $n$, the case $j=n-1$ of~\eqref{eq:dual-generated-FSFOM} gives $\Lambda_n=\frac{D}{M(N+1)^{3/2}W_{n,n-1}}$, the quantity appearing throughout~\eqref{eq:lambda(W)}. Inducting on $j$ from $n-2$ to $0$, suppose $\lambda(W)_{i,n}=\lambda_{i,n}$ for $i=j+2,\dots,n-1$. Then
$$\lambda(W)_{j+1,n}=\frac{1}{W_{j+1,j}}\left(\frac{DW_{n,j}}{M(N+1)^{3/2}W_{n,n-1}}-\frac{D}{M(N+1)^{3/2}}-\sum_{i=j+2}^{n-1}\lambda_{i,n}W_{i,j}\right)=\lambda_{j+1,n},$$
by $\Lambda_n W_{n,j}-\frac{D}{M(N+1)^{3/2}}=\sum_{i=j+1}^{n-1}\lambda_{i,n}W_{i,j}$. Finally, $\lambda(W)_{0,n}=\Lambda_n-\frac{1}{N+1}-\sum_{i=1}^{n-1}\lambda_{i,n}=\lambda_{0,n}$.

Now suppose $W$ has positive subdiagonal and set $\lambda=\lambda(W)$. Summing the definition of $\lambda_{0,n}$ in~\eqref{eq:lambda(W)} with $\frac{1}{N+1}+\sum_{i=1}^{n-1}\lambda_{i,n}$ gives $\Lambda_n=\frac{D}{M(N+1)^{3/2}W_{n,n-1}}$. We show $W(\lambda)=W$ by induction on $n$. For $n=1$, $W(\lambda)_{1,0}=\frac{D}{M(N+1)^{3/2}\Lambda_1}=W_{1,0}$. Assume $W(\lambda)_{i,j}=W_{i,j}$ for all $i<n$ and consider $j<n$. If $j=n-1$, then $W(\lambda)_{n,n-1}=\frac{D}{M(N+1)^{3/2}\Lambda_n}=W_{n,n-1}$. Otherwise,
\begin{align*}
W(\lambda)_{n,j}&=\frac{1}{\Lambda_n}\left(\frac{D}{M(N+1)^{3/2}}+\sum_{i=j+1}^{n-1}\lambda_{i,n}W_{i,j}\right)\\
&=\frac{1}{\Lambda_n}\left(\frac{D}{M(N+1)^{3/2}}+\sum_{i=j+2}^{n-1}\lambda_{i,n}W_{i,j}+\lambda_{j+1,n}W_{j+1,j}\right)=W_{n,j},
\end{align*}
where the last equality considers the definition of $\lambda_{j+1,n}$ in~\eqref{eq:lambda(W)}.
\end{proof}

\subsection{Main Theorem}

Below we characterize the entire set of minimax optimal methods $W$ as exactly those whose associated upper triangular $\lambda$ block lies in the polyhedron
\begin{equation}\label{eq:opt-lambda-set}
       \mathcal{P}^N = \left\{\lambda \mid \lambda\geq 0,\ \sum_{i=j+1}^N\lambda_{j,i} - \sum_{i=0}^{j-1}\lambda_{i,j}=\frac{1}{N+1},\ j=0,\dots,N-1\right\}.
   \end{equation}
We prove this characterization in two steps. Section~\ref{sec:sufficiency} shows that each $\lambda\in \mathcal{P}^N$ (combined with the common values~\eqref{eq:common-values}) provides a certificate of optimality for $W(\lambda)$. These proofs are exactly of the form generated by the dual in the constructive approach of~\cite{constructive_approach}. Section~\ref{sec:necessity} shows that for any optimal method $W$, any associated optimal proof multipliers (from Lemma~\ref{lem:strong-duality}) must take the common values~\eqref{eq:common-values} and set the remaining block of $\lambda$ on $i,j=0,\dots,N$ to be the upper triangular $\lambda(W)$. As a by-product, their optimal proof is unique (up to vacuous freedom in $\lambda_{i,i}$). This is proven by a novel use of complementary slackness between minimax optimal methods and maximin hard instances.

\begin{theorem}\label{thm:characterization}
   For any $M, D > 0$ and $N \in \N$, the set of optimal methods is $\Wopt= W(\mathcal{P}^N)$.
\end{theorem}

\begin{corollary}\label{cor:uniqueness}
    For $W\in \Wopt$, there exist unique dual optimal multipliers $\sigma,\lambda,\mu,Z$, given by~\eqref{eq:common-values} with the remaining $\lambda$ block on $i,j=0,\dots,N$ as the upper triangular matrix $\lambda(W)$.
\end{corollary}

Lastly, we note two connections and insights from relating these results to prior work.
Expanding the recursive form~\eqref{eq:lambda(W)} shows the multipliers $\lambda(W)$ are equal to
\begin{align}
    \lambda_{j,n}&=\frac{D}{M(N+1)^{3/2}}
\sum_{\ell\ge1}(-1)^{\ell-1}\sum_{n=k_0>k_1>\dots>k_\ell=j}
\frac{\prod_{r=1}^{\ell}\left(W_{k_{r-1},k_r-1}-W_{k_{r-1},k_r}\right)}{\prod_{r=0}^{\ell}W_{k_r,k_r-1}}, \quad 1 \leq j < n \leq N \nonumber \\
\lambda_{0,n}&=\frac{D}{M(N+1)^{3/2}}\sum_{\ell\ge0}(-1)^{\ell}\sum_{n=k_0>k_1>\dots>k_\ell\geq 1}\frac{\prod_{r=1}^{\ell}\left(W_{k_{r-1},k_r-1}-W_{k_{r-1},k_r}\right)}{\prod_{r=0}^{\ell}W_{k_r,k_r-1}}-\frac{1}{N+1}. \label{eq:lambda_poly_form}
\end{align}

The nonnegativity of these values gives the necessary polynomial inequalities in $W$, mirroring the $H$-certificates of~\cite{yoon2025}. The $H$-invariances of~\cite{yoon2025} provided further necessary equalities for optimal methods that fix the performance of each optimal method on the hard instance of~\cite{ParkRyu2022_exact}. Similarly, we find $W$-invariances for subgradient methods that fix the last row of any optimal method. Namely, from the equalities for $\lambda(W)\in\mathcal{P}^N$, one can derive\footnote{Notice that~\eqref{eq:opt-lambda-set} ensures that $\Lambda_i=\sum_{n=i+1}^N\lambda_{i,n}$ for $i<N$ and $\Lambda_N=1$. Then summing~\eqref{eq:dual-generated-FSFOM} over $n=j+1,\ldots,N$, $$\sum_{n=j+1}^N\Lambda_nW_{n,j} = (N-j)D/[M(N+1)^{3/2}] + \sum_{n=j+1}^N\sum_{i=j+1}^{n-1}\lambda_{i,n}W_{i,j}.$$
Subtracting the common $\sum_{n=j+1}^{N-1}\Lambda_nW_{n,j}$ terms from both sides and noting $\Lambda_N=1$ gives the claim.} the equalities that 
\begin{equation} \label{eq:W-invariance}
    W_{N,j} = \frac{D}{M\sqrt{N+1}}\frac{N-j}{N+1}, \qquad 0 \leq j \leq N-1.
\end{equation}

Following~\cite{yoon2025}, Yoon and Grimmer~\cite{yoon2026} developed a composition theory, which provided a combinatorial handle on extremal optimal fixed-point methods. Future work may be able to derive similar subgradient method theory by studying $\mathcal{P}^N$. (When mapped to the polyhedron $\mathcal{P}^N$, one can verify the averaged subgradient method and the momentum-type method of~\cite{constructive_approach} correspond to extreme points while the optimal subgradient method without averaging of~\cite{zamani2025exact} corresponds to the barycenter of the vertices.)

Zamani and Glineur~\cite{zamani2025exact} established that no method iterating $x_n=x_{n-1} - h_{n-1}g_{n-1}$ can guarantee optimal performance at every $n$. That is, optimal fixed $h_{0},\dots,h_{N-1}$ must depend on the horizon $N$. A similar calculation from Theorem~\ref{thm:characterization} lets us widen their conclusion to fixed-step methods: from the $W$-invariance~\eqref{eq:W-invariance}, the only possible anytime fixed-step method must iterate
\begin{equation}\label{eq:possible-anytime-method-form}
    x_n = x_0 - \sum_{i=0}^{n-1} \frac{D}{M\sqrt{n+1}}\frac{n-i}{n+1} g_i
\end{equation}
for $n=1,2,\dots$. Taking the associated $W$ matrix with $W_{n,i} = \frac{D}{M\sqrt{n+1}}\frac{n-i}{n+1}$ up to $N=5$, $\lambda(W)\not\in \mathcal{P}^5$. This can be verified by noting that for such methods, the induced proof $\lambda(W)$ has
$$\lambda_{0,2} = \frac{3\sqrt{3} - 2\sqrt{2} - \sqrt{N+1}}{(N+1)^{3/2}}, $$
which is negative at $N=5$. This boundary is sharp as the matrix up to $N=4$ has $\lambda(W)\in\mathcal{P}^4$.

\begin{corollary}
    For $N\geq 5$, no method $W$ has iterates $x_0,\dots,x_N$ each optimal at the $n$th iteration.
\end{corollary}

\section{Proof of Sufficiency for Theorem~\ref{thm:characterization}}\label{sec:sufficiency}

Below we prove that each method defined by $W(\lambda)$ for any $\lambda\in\mathcal{P}^N$ is optimal. Note that our proofs take the same form as those generated by the constructive approach of~\cite{constructive_approach}. Up to changes in notation, each considered $\lambda$ is part of a dual PEP solution for their associated Greedy First-Order Method. From this, each $W=W(\lambda)$ of our considered form could be constructively generated by~\cite[Corollary 1]{constructive_approach}. For simplicity, we have given a self-contained proof of their optimality.

Consider any $\lambda \in \mathcal{P}^N$ and let $W=W(\lambda)$, which is well-defined as $\Lambda_n = \frac{1}{N+1} + \sum_{i=0}^{n-1}\lambda_{i,n} > 0$ by the nonnegativity of $\lambda$.  Fix the values for $\sigma,\lambda_{\star,i},\mu_i$ as in~\eqref{eq:common-values} and define $Z=\sigma ww^\top$ with $w=\left[1,-\frac{D}{M\sqrt{N+1}},\dots,-\frac{D}{M\sqrt{N+1}}\right]^\top$.
We claim the core PEP identity holds
\begin{equation} \label{eq:core-proof-identity}
    \sigma\mathcal D+\sum_{i,j \in \Ical_N^\star}\lambda_{i,j}\mathcal C_{i,j}+\sum_{i=0}^N\mu_i\mathcal{M}_i+\inner{Z}{G} = \frac{MD}{\sqrt{N+1}} - (f_N-f_\star).
\end{equation}
Once proven, the fact that the method with stepsizes $W$ attains the optimal $\frac{MD}{\sqrt{N+1}}$ worst-case rate is immediate by the nonnegativity of the left-hand-side.

Observe that the identity~\eqref{eq:core-proof-identity} is affine in $F$ and $G$ (once one substitutes the equalities $g_\star=0$ and $x_n = x_{0} - \sum_{i=0}^{n-1} W_{n,i} g_i$). Hence, we can verify the identity by verifying that each associated coefficient matches between the two sides. The constant terms agree as their correspondence amounts to
\begin{equation}
     \sigma D^2 + \sum_{i=0}^N \mu_i M^2 = \frac{MD}{\sqrt{N+1}}
\end{equation}
which is immediate from the values of $\sigma$ and $\mu_i$ in~\eqref{eq:common-values}.

Next, for each $j$, we verify that the coefficients of each $f_j$ on the left-hand-side reduce to the coefficients on the right-hand-side. Namely, only the constraints $\mathcal{C}_{i,j}$ depend on $f$ terms, so we have
\begin{align*}
    \sum_{i, j \in \Ical_N^\star}\lambda_{i,j}(f_i-f_j) &= \sum_{j=0}^N \left(\sum_{i>j}\lambda_{j,i}-\sum_{i<j}\lambda_{i,j}-\lambda_{\star,j}\right) f_j + \sum_{j=0}^N \lambda_{\star,j} f_\star = -(f_N-f_\star)
\end{align*}
where the first equality rearranges summations, noting $\lambda_{i,\star}$ and $\lambda_{i,j}$ with $i > j$ are zero, and the second applies the defining equalities of $\lambda\in\mathcal{P}^N$ and our choice of $\lambda_{\star,i}=\frac{1}{N+1}$.

Now, we can complete our proof by verifying the coefficients agree on each term in $G$. In particular, we show each term vanishes on the left-hand-side as the right-hand-side has no dependence on $G$. Let $Z_{\star,\star}$ and $Z_{\star,n}$ denote the coefficients on $\|x_0-x_\star\|^2$ and  $\langle x_0-x_\star, g_n\rangle$ in $\inner{Z}{G}$. First, note that the coefficient on $\|x_0-x_\star\|^2$ is zero as
\begin{align*}
    \sigma(-\|x_0-x_\star\|^2) + Z_{\star,\star} \|x_0-x_\star\|^2 = \left(-\frac{M}{2D\sqrt{N+1}} + \frac{M}{2D\sqrt{N+1}}\right) \|x_0-x_\star\|^2 = 0.
\end{align*}
Considering the coefficients on $\|g_n\|^2$, we simplify the left-hand-side as follows
\begin{align*}
    \mu_n (-\|g_n\|^2) + Z_{n,n}\|g_n\|^2 = \left(-\frac{D}{2M(N+1)^{3/2}} + \frac{D}{2M(N+1)^{3/2}}\right)\|g_n\|^2 = 0.
\end{align*}
Lastly, we consider the coefficients on off-diagonal components of $G$. For each $n$, consider the terms associated with $\langle g_n, x_0-x_\star\rangle$ and $\inner{g_n}{g_j}$ with $j<n$. On these terms, the left-hand-side reduces as
\begin{align*}
&\left\langle g_n,
    \left(\lambda_{\star,n}+2Z_{\star,n}\right)(x_0-x_\star)
    +\sum_{j=0}^{n-1}\left(
        \sum_{i=j+1}^{n-1}\lambda_{i,n}W_{i,j}
        -\Lambda_n W_{n,j}
        +2Z_{j,n}
    \right)g_j
\right\rangle \\
&=
\left\langle g_n,
    \sum_{j=0}^{n-1}\left(
        \sum_{i=j+1}^{n-1}\lambda_{i,n}W_{i,j}
        -\Lambda_n W_{n,j}
        +\frac{D}{M(N+1)^{3/2}}
    \right)g_j
\right\rangle =0
\end{align*}
where the first equality applies $\lambda_{\star,n}+2Z_{\star,n} = 0$, and the second equality notes these terms vanish as $W=W(\lambda)$ satisfying~\eqref{eq:dual-generated-FSFOM}.
Hence, each coefficient in the identity~\eqref{eq:core-proof-identity} agrees, proving the identity and $W\in\Wopt$.

\section{Proof of Necessity for Theorem~\ref{thm:characterization}}\label{sec:necessity}

Throughout this section, let $M, D > 0$ and $N \in \mathbb{N}$, and fix some $W  \in \Wopt$. 
First, we recall a well-known hard problem instance for subgradient methods~\cite[Theorem A.1]{drori2016optimal}:
For any $k\in\mathbb{N}$, let $e_1, \dots, e_k$ denote an orthonormal basis for $\mathbb{R}^k$. We define a corresponding adversarial problem instance as \begin{equation}\label{eq:hard-instance}
    f(x) = \max\left\{\max_{1 \leq i \leq k} \inner{x}{Me_i}, -\frac{MD}{\sqrt{k}}\right\}, \quad x_0 = 0, \quad
g(x) = \begin{cases}
Me_{i_\star} & \text{if } f(x) > -\frac{M D}{\sqrt{k}}\\
0 & \text{otherwise},
\end{cases}
\end{equation}
where $i_\star = \min\{i : f(x) = \inner{x}{M e_i}\}$. Notably, $f$ is convex and $M$-Lipschitz, minimizer $x_\star = -\frac{D}{\sqrt{k}}[1, \dots, 1]^\top$ satisfies  $\|x_0 - x_\star\| = D$, and $f(x_\star) = -\frac{MD}{\sqrt{k}}$. For $k=N+1$, no $N$-step subgradient span method can produce a point with $f(x)$ less than zero. Hence, this instance proves the minimax optimal rate's lower bound.
By considering the hard instance with $k=N$, we further find that any optimal $W$ must have strictly positive lower triangular entries.
\begin{lemma}\label{lem:opt-W-positive}
    $W \in \mathbb{L}^{N}_{++}$.
\end{lemma}
\begin{proof}
Suppose to the contrary that $W_{n,j} \leq 0$ for some $j \leq n-1$. In particular, let $\hat{n}$ denote the earliest iteration where such a nonpositive value $W_{\hat{n},j}$ occurs.
We then derive a contradiction by showing that when applied to the hard problem instance~\eqref{eq:hard-instance} with $k = N$, this $\hat{n}$th iteration will cause the method $W$ to be strictly suboptimal. 

First, we consider the case when $\hat{n} < N$. Notice that for any iteration $n < \hat{n}$, since $W_{n,j}>0$, the finite maximum defining $f$ at $x_{n}$ is inactive on its first $n$ components. Hence, the adversarial oracle will set $g_n = Me_{n+1}$.
However,  at iteration $\hat{n}$, the oracle will set $g_{\hat{n}} = M e_{j+1}$ for some previous $j < \hat{n}$ with nonpositive $W_{\hat{n},j}$. As a result, $g_{\hat{n}}$ lies in the span of previously seen subgradients.
Therefore, at iteration $N$, the span of $\{g_i\}_{i=0}^{N-1}$ contains at most $N-1<k$ of the orthonormal directions. Consequently, $\langle e_i,x_N\rangle=0$ for some $i$, causing $f(x_N)\geq 0$. In the case where $\hat{n} = N$, some $j$ has $\inner{x_N}{e_{j+1}}=-MW_{N,j} \geq 0$. Either way, $f(x_N) - f(x_\star) \geq 0 - (-\frac{MD}{\sqrt{N}}) > \frac{MD}{\sqrt{N+1}}$, which contradicts optimality.
\end{proof}

Since $W$ has positive lower triangular entries, Lemma~\ref{lem:strong-duality} applies. Thus we can fix a selection of nonnegative $\sigma, \lambda_{i,j}, \mu_i$ and positive semidefinite $Z$ such that the certificate~\eqref{eq:certificate} holds. Since  $W \in \Wopt$, we have that $r = p_\star(W) = \frac{MD}{\sqrt{N+1}}$.

Continuing our examination of the known hard instance with $k=N+1$, we find that structure is forced among these dual values as well via complementary slackness. 
As a shorthand, let $(F^W,G^W)$ denote the function values and Gram inner products observed when $W$ is applied to the hard instance~\eqref{eq:hard-instance}. We denote the realized constraints by $\mathcal{D}^W := \mathcal{D}(F^W, G^W)$,  $\mathcal{C}_{i,j}^W := \mathcal{C}_{i,j}(F^W, G^W)$,  and $\mathcal{M}_i^W := \mathcal{M}_i(F^W, G^W)$.

\begin{lemma}\label{lem:comp_slackness}
   The dual optimal multipliers $\lambda_{i,j}$ for the constraints $\mathcal{C}_{i,j}\geq 0$ satisfy $\lambda_{i,j}= 0$ for all $i > j$ and $\lambda_{i,\star} = 0$ for all $i = 0, \dots, N$.
\end{lemma}

\begin{proof}
 Since $W \in \mathbb{L}^{N}_{++}$ by Lemma~\ref{lem:opt-W-positive}, the iterate $x_n$ produced by $W$ will always have $f(x_n)$ inactive on its first $n$ components. Thus each iteration $n$ will have $f_n^W=0$ and $g_n^W = Me_{n+1}$. The following values are then forced, $x_0^W-x_\star^W  = \frac{D}{\sqrt{N+1}}[1, \dots, 1]^\top$ and $x_n^W = -\sum_{i=0}^{n-1}W_{n,i}g_i^W$. Plugging these into $\mathcal{C}_{i,j}$ shows that
 \begin{align*}
        \mathcal{C}^W_{i,j} &= -\inner{g_j^W}{x_i^W} = \begin{cases}
            M^2W_{i,j} & i > j\\
            0 & \text{else}
        \end{cases}, \quad \mathcal{C}^W_{i,\star} = -f_\star^W = \frac{MD}{\sqrt{N+1}}, \quad \mathcal{C}_{\star,i}^W = f_\star^W - \inner{g_i^W}{x_\star^W} = 0.
    \end{align*}
    Since $W \in \mathbb{L}^{N}_{++}$, the slacks $\mathcal{C}^W_{i,j}$ for $i>j$ and $\mathcal{C}^W_{i,\star}$ are strictly positive. Since the right-hand-side of~\eqref{eq:certificate} is exactly zero on this tight instance, each nonnegative term on the left-hand-side must be zero. Hence, by complementary slackness, $\lambda_{i,j}= 0$ for $i > j$ and $\lambda_{i,\star} = 0$.
\end{proof}

\begin{lemma}\label{lem:rank-one-Z}
 The dual-optimal multiplier $Z\succeq0$ for the constraint $G\succeq0$ satisfies $Z=\sigma\,ww^\top$, where $\sigma\geq0$ is the dual multiplier of $\mathcal D$ and $w = [1, -\frac{D}{M\sqrt{N+1}}, \dots, -\frac{D}{M\sqrt{N+1}}]^\top.$
\end{lemma}
\begin{proof}
We continue considering the values $g^W_n = M e_{n+1}$ and $x_0^W-x_\star^W  = \frac{D}{\sqrt{N+1}}[1, \dots, 1]^\top$ derived in the previous lemma. Observe that these subgradients are orthogonal and $x_0^W-x_\star^W$ lies in their span. So $\mathrm{rank}(G^W)=N+1$ and $\ker G^W=\spann{w}$. By the same complementary slackness reasoning, we must have $\inner{Z}{G^W}=0$.
Since $Z,G^W\succeq0$, it holds that $Z G^W=0$, making $\mathrm{range}(G^W)\subseteq\ker Z$.

Noting $\mathrm{range}(G^W)$ has codimension one, $\mathrm{rank}(Z) \leq 1$, and in particular, $\mathrm{range}(Z)\subseteq\ker G^W=\spann w$. Therefore, $Z = \beta w w^\top$ for some $\beta \geq 0$. 
Considering the identity~\eqref{eq:certificate},
the coefficient on $\|x_0-x_\star\|^2$ is independent of the constraints $\mathcal C_{i,j}$ and $\mathcal{M}_i$. Hence, $\sigma\mathcal D=\sigma(D^2-\|x_0-x_\star\|^2)$ is the sole source of $\|x_0-x_\star\|^2$ besides $\inner{Z}{G}$. Thus $-\sigma+Z_{\star,\star}=0$ and so  $Z=\sigma w w^\top$.
\end{proof}

By considering the core PEP identity~\eqref{eq:certificate}, restated below for ease, we will find the multiplier values for $W$ are uniquely determined and, in particular, the corresponding $\lambda$ block must lie in $\mathcal{P}^N$ and equal $\lambda(W)$, completing our proof. Recall that every optimal proof certificate satisfies
$$\sigma\mathcal D+\sum_{i,j \in \Ical_N^\star}\lambda_{i,j}\mathcal C_{i,j}+\sum_{i=0}^N\mu_i\mathcal{M}_i+\inner{Z}{G} = \frac{MD}{\sqrt{N+1}} - (f_N-f_\star).$$  By Lemma~\ref{lem:comp_slackness}, $\lambda_{i,j}=0$ for $i>j$ and $\lambda_{i,\star}=0$, so the only remaining convexity terms are $\lambda_{i,j}$ for $i<j$ and $\lambda_{\star,i}$. 

Substituting $x_n=x_0-\sum_{i=0}^{n-1}W_{n,i}g_i$, and expanding $\inner{g_i}{x_i-x_\star}=\inner{g_i}{x_0-x_\star}-\sum_{k=0}^{i-1}W_{i,k}\inner{g_i}{g_k}$ in each $\mathcal{C}_{\star,i}$, the part of $\sigma\mathcal D+\sum_{i,j}\lambda_{i,j}\mathcal C_{i,j}+\sum_i\mu_i\mathcal M_i$ depending on $G$ is
$$-\sigma\|x_0-x_\star\|^2-\sum_{i=0}^N\mu_i\|g_i\|^2+\sum_{i=0}^N\lambda_{\star,i}\inner{g_i}{x_0-x_\star}+\sum_{j<n}\left(\sum_{i=j+1}^{n-1}\lambda_{i,n}W_{i,j}-\Lambda_n W_{n,j}\right)\inner{g_n}{g_j},$$
where $\Lambda_n = \lambda_{\star,n}+\sum_{i=0}^{n-1}\lambda_{i,n}$.  Considering the form of $Z$ from Lemma~\ref{lem:rank-one-Z}, writing
$G=P^\top P$ for $P=[x_0-x_\star \ | \ g_0 \ |\  g_1 \ |\  \dots\ | \ g_N]$, we note  
\begin{align*}
\inner{Z}{G}&=\sigma\left\|x_0-x_\star-\frac{D}{M\sqrt{N+1}}\sum_{i=0}^N g_i\right\|^2.
\end{align*} Adding this term, and noting each Gram coordinate's coefficient vanishes in~\eqref{eq:certificate}, we conclude
\begin{align}\label{eq:G-values}
\lambda_{\star,i}=2\sigma\frac{D}{M\sqrt{N+1}}, \qquad \mu_i=\sigma\left(\frac{D}{M\sqrt{N+1}}\right)^2, \qquad \Lambda_n W_{n,j}=\sum_{i=j+1}^{n-1}\lambda_{i,n}W_{i,j}+2\sigma\left(\frac{D}{M\sqrt{N+1}}\right)^2,
\end{align}
by considering the coefficients on the $\inner{g_i}{x_0-x_\star}$, $\|g_i\|^2$, and  $\inner{g_n}{g_j}$ respectively.

Considering the coefficients on the $f_i$ terms, it follows that
\begin{align*}
    \sum_{j=0}^N \left(\sum_{i>j}\lambda_{j,i}-\sum_{i<j}\lambda_{i,j}-\lambda_{\star,j}\right) f_j + \sum_{i=0}^N \lambda_{\star,i} f_\star +(f_N-f_\star) = 0.
\end{align*}
Setting each coefficient to zero requires 
\begin{align}\label{eq:f-values}
\sum_{i>j} \lambda_{j,i}-\sum_{i<j} \lambda_{i,j} = \lambda_{\star,j}, \ \ \forall j = 0, \dots, N-1,
\qquad \sum_{i < N}\lambda_{i,N} = 1 - \lambda_{\star,N}, \qquad 
\sum_{i=0}^N \lambda_{\star,i} = 1
\end{align}
by considering the $f_i$ with $i < N$, $f_N$ and $f_\star$ terms respectively.

Finally, comparing constant terms in the identity, $\frac{MD}{\sqrt{N+1}} = \sigma D^2 + M^2 \sum_{i=0}^N \mu_i. $
From the values of $\lambda_{\star,i}$ and $\mu_i$ above, we conclude that $\sigma = \frac{M}{2D\sqrt{N+1}}$, $\lambda_{\star,i} = \frac{1}{N+1}$ and $\mu_i = \frac{D}{2M(N+1)^{3/2}}$. With these values of $\lambda_{\star,i}$, the constraints on $\lambda$ in~\eqref{eq:f-values}, along with nonnegativity from Lemma~\ref{lem:strong-duality}, collapse to the constraints defining $\mathcal{P}^N$. Therefore $\lambda \in \mathcal{P}^N$. Lastly, considering the value of $\sigma$, from~\eqref{eq:G-values}, $W=W(\lambda)$ is generated from~\eqref{eq:dual-generated-FSFOM}, so $W \in W(\mathcal{P}^N)$.

To prove Corollary~\ref{cor:uniqueness}, observe that we have established the uniqueness of the dual multipliers, fixing $\lambda_{i,j}$ for $0 \leq i < j \leq N$ and setting the rest as in~\eqref{eq:common-values}. Then, since $W = W(\lambda)$ and $W(\cdot)$ is a bijection onto ${\Lambda_n > 0}$ by Proposition~\ref{prop:bijection}, $\lambda = \lambda(W)$.

    \paragraph{Acknowledgments.} This work was supported in part by the Air Force Office of Scientific Research. Benjamin Grimmer was additionally supported as a fellow of the Alfred P. Sloan Foundation.
  
    {\small
    \bibliographystyle{plain}
    \bibliography{bibliography}
    }

\end{document}